\documentclass[11pt]{article}
\usepackage[latin1]{inputenc}
\usepackage{epsfig}
\usepackage{color}
\usepackage[british,english]{babel}
\usepackage{amsthm}
\usepackage{amsmath}
\usepackage{amsfonts}
\usepackage{amssymb}
\usepackage{graphicx}
\usepackage{float}
\newtheorem{definition}{Definition}

\newtheorem{lemma}{Lemma}
\newtheorem{proposition}{Proposition}

\newtheorem{theorem}{Theorem}
\newfont{\sBlackboard}{msbm10 scaled 900}

\newcommand{\mylabel}[1]{\label{#1}
            \ifx\undefined\stillediting
            \else \fbox{$#1$}\fi }
\newcommand{\BE}{\begin{equation}}

\newcommand{\EEQ}{\end{equation}}
\newcommand{\rfb}[1]{\mbox{\rm
   (\ref{#1})}\ifx\undefined\stillediting\else:\fbox{$#1$}\fi}

\newfont{\Blackboard}{msbm10 scaled 1200}

\newfont{\roma}{cmr10 scaled 1200}

\def\CC{\rm \hbox{C\kern-.56em\raise.4ex
         \hbox{$\scriptscriptstyle |$}\kern+0.5 em }}

\def\n{|\kern -.05cm{|}\kern -.05cm{|}}

\def\R{{\bf \hbox{\sc I\hskip -2pt R}}} 

\newcommand{\mm}    {{\hbox{\hskip 0.5pt}}}

\newcommand{\bluff} {{\hbox{\raise 15pt \hbox{\mm}}}}
\makeatletter
\def\section{\@startsection {section}{1}{\z@}{-3.5ex plus -1ex minus
    -.2ex}{2.3ex plus .2ex}{\large\bf}}
\makeatother
\def\be{\begin{equation}}
\def\ee{\end{equation}}

\date{ }
\begin{document}
\thispagestyle{empty}
\title{\bf Minimaxity and Limits of Risks Ratios of Shrinkage Estimators of a Multivariate Normal Mean in the Bayesian Case  }\maketitle
\author{ \center Abdenour Hamdaoui$^{1}$, Abdelkader Benkhaled$^{2}$ and Nadia Mezouar$^{3}$\\}
$^{1}$ University of Science and Technology of Oran, Mohamed Boudiaf (USTOMB), Department of Mathematics; Laboratory of Statistics and Random Modelisations of University of Tlemcen, Algeria.\\
$^{2}$ Department of Biology, Mascara University, Algeria.\\
$^{3}$ Faculty of Economics and Commercial Sciences, Mascara University, Algeria.\\

\renewcommand{\abstractname} {\bf Abstract}
\begin{abstract}
In this article, we consider two forms of shrinkage estimators of the mean $\theta$ of a multivariate normal distribution $X\sim N_{p}\left(\theta, \sigma^{2}I_{p}\right)$ where $\sigma^{2}$ is unknown. We take the prior law $\theta \sim N_{p}\left(\upsilon, \tau^{2}I_{p}\right)$ and we constuct a Modified Bayes estimator $\delta_{B}^{\ast}$ and an Empirical Modified Bayes estimator $\delta_{EB}^{\ast}$. We are interested in
studying the minimaxity and the limits of risks ratios of these estimators, to the maximum likelihood estimator $X$, when $n$ and $p$ tend to infinity.
\end{abstract}
\noindent
{\small \bf AMS classification numbers:} 62F15, 62J07.

\newpage
\section {Introduction}
The problem of estimating the mean vector $\theta$ of a multivariate normal distribution $N_{p}(\theta, \sigma^{2}I_{p})$ in
$\R^{p}$, has experienced many development since the papers \cite{s56, js61, ls72}. In these works one estimates the mean $\theta$ by shrinkage estimators deduced from the empirical mean estimator, which are better in quadratic loss than the empirical mean estimator. The idea of Stein \cite{s56} showed that the maximum likelihood estimator of the mean $\theta$ is inadmissible when the dimension of space parameters exceeds two. James and Stein \cite{js61}, provided a constructive shrinkage estimator denoted by $\delta^{JS}=(1-(p-2)S^{2}/(n+2)\| X \|^{2}) X$, which dominates the empirical mean estimator $\delta_{0}=X$, when the dimension of the parameter space $p$ is $\geq3$.
Baranchik \cite{b64}, proposed the positive-part of the James-Stein estimator $\delta^{JS+}=\max (0;1-(p-2)S^{2}/(n+2)\| X \|^{2}) X $, an estimator dominating the James-Stein estimator.\\
When the dimension $p$ is infinite, Casella and Hwang \cite{ch82}, studied the case where $\sigma^{2}$ is known $(\sigma^{2}=1)$ and showed that if the limit of the ratio $\|\theta\|^{2}/p$ is a constant $c>0$, then the risks ratios of the James-Stein estimator $\delta^{JS}$ and the positive-part of the James-Stein estimator $\delta^{JS+}$, to the maximum likelihood estimator $X$, tend to a constant value $c/(1+c)$. Benmansour and Hamdaoui \cite{bh11} have taken the same model given by Casella and Hwang \cite{ch82}, where the parameter $\sigma^{2}$ is unknown and they established the same results. Hamdaoui and  Benmansour \cite{hb15}, considered the model $X\sim N_{p}(\theta, \sigma^{2}I_{p})$ where $\sigma^{2}$ is unknown and estimated by $S^{2}$ ($S^{2}\sim \sigma^{2}\chi_{n}^{2}$).  They studied the following class of shrinkage estimators $\delta_{\phi }=\delta^{JS}+l(S^{2}\phi( S^{2}, \| X \|^{2})/\|X\|^{2})X$. The authors showed that, when the sample size $n$ and the dimension of parameter space $p$ tend to infinity, the estimators $\delta_{\phi }$ have a lower bound $B_{m}=c/(1+c)$ and if the shrinkage function $\phi$ satisfies some conditions, the risks ratio $R(\delta_{\phi}, \theta)/R(X, \theta)$ attains this lower bound $B_{m}$, in particulary the risks ratios $R(\delta^{JS}, \theta)/R(X, \theta)$ and $R(\delta^{JS+}, \theta)/R(X, \theta)$. In Hamdaoui et al \cite{ham16}, the authors studied the limit of risks ratios of two forms of shrinkage estimators. The first one has been introduced by Benmansour and Mourid \cite{bm07}, $\delta_{\psi}=\delta^{JS}+l(S^{2}\psi( S^{2},\| X \|^{2})/\| X \|^{2})X$, where $\psi(.,u)$ is a function with support $[0,b]$ and satisfies some conditions differents from the one given in Hamdaoui and Benmensour \cite{hb15}. The second is the polynomial form of shrinkage estimator introduced by Li and Kio \cite{tzi82}. Hamdaoui and Mezouar \cite{hm17} studied the general class of shrinkage estimators $\delta_{\phi} =(1-S^{2}\phi ( S^{2}, \| X \|^{2})/\| X \|^{2}t) X$. They showed the same results given in Hamdaoui and  Benmansour \cite{hb15}, with different conditions on the shrinkage function $\phi$.

When the dimension $p$ is finite, many authors studied the minimaxity of shrinkage estimators of a multivariate normal mean, see for example \cite{bs12, m14, s16}.

In this paper, we extend our previous works to the Bayesian case. We adopt the model $X\sim N_{p}\left(\theta, \sigma^{2}I_{p}\right)$ and independently of the observations $X$, we observe $S^{2}\sim \sigma^{2}\chi_{n}^{2}$ an estimator of $\sigma^{2}$. We consider the prior distribution $\theta \sim N_{p}\left(\upsilon, \tau^{2}I_{p}\right)$ where the hyperparameter $\upsilon$ is known and the hyperparameter $\tau ^{2}$ is known or unknown. Note that
$R(X,\theta )=p\sigma ^{2},$ is the quadratic risk of the maximum likelihood estimator. It is well known that the maximum likelihood estimator $X$ is
minimax, so that any estimator dominating it is also minimax. Our goal is to estimate the mean $\theta$ by a Modified Bayes estimator $\delta_{B}^{\ast}$ when the hyperparameter $\tau^{2}$ is known and by an Empirical Modified Bayes estimator $\delta_{EB}^{\ast}$ when the hyperparameter $\tau^{2}$ is unknown.\\
The paper is organized as follows. In Section 1, we give preliminaries containing some results used in the next sections. In Section 2, we give the main results of this paper. First, we take the prior law of $\theta: \theta \sim N_{p}(\upsilon, \tau^{2}I_{p})$ where the hyperparameters $\upsilon$ and $\tau^{2}$ are known and we construct a Modified Bayes estimator $\delta_{B}^{\ast}$. When $n$ and $p$ are fixed, we show that the estimator $\delta_{B}^{\ast}$ is minimax. We study the behaviour of the risks ratio of this estimator to the Maximum likelihood estimator $X$, when $n$ and $p$ tend simultaneously to infinity without assuming any order relation or functional relation between $n$ and $p$. In the second part of this section, we take the prior distribution of $\theta: \theta \sim N_{p}(\upsilon, \tau^{2}I_{p})$ where the
hyperparameter $\upsilon$ is known and the hyperparameter $\tau^{2}$ is unknown and we construct an Empirical Modified Bayes estimators $\delta_{EB}^{\ast}$ of the mean $\theta$. We will follow the same steps as have been given in first part. In the third part of this section, we illustrate graphically the results given in the paper. Finally, we give an Appendix containing technical lemmas used in the proofs of our results.

\section{Preliminaries}
We recall that if $X$ is a multivariate Gaussian random $N_{p}(\theta, \sigma^{2}I_{p}) $ in $\R^{p}$, then $\frac{\| X\|^{2}}{\sigma^{2}}\sim \chi_{p}^{2}(\lambda )$ where $\chi_{p}^{2}(\lambda )$ denotes the non-central chi-square distribution with $p$ degrees of freedom and non-centrality parameter $\lambda=\frac{\| \theta \|^{2}}{2\sigma^{2}}$. The following definition is used to calculate the expectation of functions of a non-central chi-square law's variable.\\

\begin{definition} \label{de 2.1}
Let $U \sim \chi_{p}^{2}(\lambda)$. For any function $f: \R_{+} \longrightarrow \R$, $\chi_{p}^{2}(\lambda )$ integrable, we have
\begin{eqnarray*}
E[f(U)]&=& E_{\chi_{p}^{2}(\lambda )}[f(U)]
\\&=& \int_{\R}f(u)\chi_{p}^{2}(\lambda )du
\\&=&\sum_{k=0}^{+\infty }[\int_{\R_{+}}f(u)\chi_{p+2k}^{2}(0) du] e^{-\frac{\lambda}{2}}\frac{(\frac{\lambda}{2})^{k}}{k!}
\\&=&\sum_{k=0}^{+\infty }[\int_{\R_{+}}f(u)\chi_{p+2k}^{2}du] P( \frac{\lambda }{2};dk) ,
\end{eqnarray*}
where $P( \frac{\lambda }{2};dk)$ being the Poisson distribution of parameter $\frac{\lambda }{2}$ and $\chi_{p+2k}^{2}$ is
the central chi-square distribution with $p+2k$ degrees of freedom.
\end{definition}
We recall the following Lemma given by Fourdrinier et al. \cite{fos08}, that we will use often in the next.\\

\begin{lemma} \label{l 2.1}
Let $X\sim N_{p}(\theta, \sigma^{2}I_{p})$ with $\theta \in \R^{p}$. Then\\
a) for $p\geq 3$, we have: $E(\frac{1}{\| X\|^{2}})=\frac{1}{\sigma^{2}}E(\frac{1}{p-2+2K})$,\\
b) for $p\geq 5$, we have: $E(\frac{1}{(\| X\|^{2})^{2}})=\frac{1}{\sigma^{4}}E(\frac{1}{(p-2+2K)(p-4+2K)})$,\\
where $K\sim P(\frac{\| \theta \|^{2}}{2\sigma^{2}})$ being the Poisson distribution of parameter $\frac{\|\theta \|^{2}}{2\sigma^{2}}.$
\end{lemma}

Now, we recall some known results of Bayes estimator.\\
Let $X|\theta \sim N_{p}(\theta, \sigma^{2}I_{p})$ and $\theta \sim N_{p}(\nu, \tau^{2}I_{p})$ where $\sigma^{2}$ is known, and hyperparameters $\nu, \tau^{2}$ are known. From  Lindley et al. \cite{ls72}, we have
\begin{equation*}
\theta| X \sim N_{p}(\nu+B( X-\nu ), \sigma^{2}BI_{p}) \, \text{where}\, B=\frac{\tau^{2}}{\tau^{2}+\sigma^{2}}.
\end{equation*}
Then, the Bayes estimator of $\theta$ is
\begin{equation*}
\delta_{B}(X)=E(\theta \vert X )=\nu+B(X-\nu),
\end{equation*}
thus
\begin{eqnarray}\label{eq2.1}
\delta_{B}(X)=(1-\frac{\sigma^{2}}{\tau^{2}+\sigma^{2}})(X-\nu)+\nu.
\end{eqnarray}
We deduce that
\begin{equation*}
i)\, R(\delta_{B}(X); \theta)=(1-B)^{2}\|\theta -\nu \|^{2}+B^{2}p \sigma^{2}
\end{equation*}
\begin{equation*}
ii)\, \frac{R(\delta_{B}(X);  \nu, \tau^{2}, \sigma^{2})}{R(X)}=\frac{\tau^{2}}{\tau^{2}+\sigma^{2}}.
\end{equation*}

\section{Main results}
In this section we are interested in studying the minimaxity, bounds and limits of risks ratios of a Modified Bayes estimator and an Empirical Modified Bayes estimator, to the maximum likelihood estimator $X$.

To proof our main results we give the following Lemmas.

\begin{lemma} \label{l 3.1}
Let $f$  is a real function. If for $p\geq 3$, $E_{\chi_{p}^{2}(\lambda)}[f(U)]$ exists, then\\
a) if $f$ is monotone non-increasing we have
\begin{equation*}
E_{\chi_{p+2}^{2}(\lambda)}[f(U)]\leq E_{\chi_{p}^{2}(\lambda )}[f(U)],
\end{equation*}
b) if $f$ is monotone non-decreasing we have
\begin{equation*}
E_{\chi_{p+2}^{2}(\lambda )}[f(U)] \geq E_{\chi_{p}^{2}(\lambda)}[f(U)] .
\end{equation*}
\end{lemma}
\begin{proof}
a) From the Definition \ref{de 2.1}, we have
\begin{equation*}
E_{\chi_{p+2}^{2}( \lambda )}[f(U)]-E_{\chi_{p}^{2}(\lambda)}[f(U)]=E[ f(\chi_{p}^{2}(\lambda))(\frac{\chi _{p}^{2}(\lambda )}{p+2K}-1)] ,
\end{equation*}
where $K\sim P( \frac{\lambda }{2}) $ being the Poisson distribution of parameter $\frac{\lambda }{2}$.\\
Using the conditional expectation and the fact that, the covariance of two functions one increasing and the other decreasing is non-positive, we obtain
\begin{eqnarray*}
E[f(\chi_{p}^{2}( \lambda ))( \frac{\chi_{p}^{2}( \lambda )}{p+2K}-1)] &=&E{ E[f(\chi_{p}^{2}( \lambda ) )( \frac{\chi _{p}^{2}( \lambda )}{p+2K}-1) ]|K}  \\
&\leq &E[f(\chi_{p}^{2}( \lambda))] \times E[( \frac{\chi_{p}^{2}( \lambda )}{p+2K}-1)]
\end{eqnarray*}
and
\begin{eqnarray*}
E[( \frac{\chi _{p}^{2}( \lambda )}{p+2K}-1)]&=&E{ E[( \frac{\chi_{p}^{2}( \lambda )}{p+2K}-1)] |K}  \\
&=&E{ E[( \frac{p+2K}{p+2K}-1)] |K}  \\
&=&0,
\end{eqnarray*}
the penultimate equality follows from the Definition \ref{de 2.1}. Thus
\begin{equation*}
E_{\chi_{p+2}^{2}( \lambda )}[f(U)]\leq E_{\chi_{p}^{2}(\lambda )}[f(U)].
\end{equation*}
In the same way, we get b).
\end{proof}

\begin{lemma} \label{l 3.2}
For any $c>0$, we have
\begin{eqnarray}\label{eq3.1}
\frac{1}{n+2+c}\leq E_{\chi_{n+2}^{2}}( \frac{1}{u+c}) \leq \frac{1}{n+c},
\end{eqnarray}
\begin{eqnarray}\label{eq3.2}
\frac{1}{( n+4+c)^{2}}\leq E_{\chi_{n+4}^{2}}[ \frac{1}{( u+c)^{2}}] \leq \frac{1}{( n+c)^{2}}.
\end{eqnarray}
\end{lemma}

\begin{proof}
On the one hand, from Jensen's inequality we have
\begin{equation*}
E_{\chi_{n+2}^{2}}( \frac{1}{u+c}) \geq \frac{1}{n+2+c}.
\end{equation*}
On the other hand, from Lemma \ref{l 5.2} (Appendix), we have
\begin{eqnarray*}
1 &=&E_{\chi_{n}^{2}}( \frac{u}{u+c}) +c\ E_{\chi_{n}^{2}}( \frac{1}{u+c})  \\
&=&n\ E_{\chi_{n+2}^{2}}( \frac{1}{u+c}) +c\ E_{\chi_{n}^{2}}( \frac{1}{u+c}) ,
\end{eqnarray*}
then
\begin{equation*}
E_{\chi_{n+2}^{2}}( \frac{1}{u+c}) =\frac{1}{n}{1-cE_{\chi_{n}^{2}}( \frac{1}{u+c})}
\end{equation*}
\begin{eqnarray}\label{eq3.3}
\leq \frac{1}{n+c},
\end{eqnarray}
the inequality (\ref{eq3.3}) follows from Jensen's inequality.

The proof of the formula (\ref{eq3.2}) is as follows: from Jensen's inequality we have
\begin{eqnarray}\label{eq3.4}
E_{\chi_{n+4}^{2}}[ \frac{1}{( u+c)^{2}}] \geq \frac{1}{( n+4+c)^{2}}.
\end{eqnarray}
In other hand, we have
\begin{equation*}
E_{\chi_{n+2}^{2}}( \frac{1}{u+c}) =E_{\chi_{n+2}^{2}}( \frac{u}{( u+c)^{2}}) +c\ E_{\chi_{n+2}^{2}}( \frac{1}{( u+c)^{2}})
\end{equation*}
\begin{eqnarray}\label{eq3.5}
\hspace{3cm}=(n+2)\ E_{\chi_{n+4}^{2}}[ \frac{1}{( u+c)^{2}}]+c\ E_{\chi_{n+2}^{2}}[ \frac{1}{( u+c)^{2}}]
\end{eqnarray}
\begin{eqnarray}\label{eq3.6}
\hspace{3cm}\geq (n+2)\ E_{\chi_{n+4}^{2}}[ \frac{1}{( u+c)^{2}}]+c\ E_{\chi_{n+4}^{2}}[ \frac{1}{( u+c)^{2}}]
\end{eqnarray}
\begin{equation*}
\geq (n+2+c)E_{\chi_{n+4}^{2}}[ \frac{1}{( u+c)^{2}}],
\end{equation*}
the equality (\ref{eq3.5}) follows from Lemma \ref{l 5.2} and the inequality (\ref{eq3.6}) follows from Lemma \ref{l 3.1}.\\
Hence
\begin{equation*}
E_{\chi_{n+4}^{2}}[ \frac{1}{( u+c)^{2}}] \leq \frac{1}{n+2+c}\ E_{\chi_{n+2}^{2}}[ \frac{1}{u+c}] .
\end{equation*}
Using the formula (\ref{eq3.1}), we obtain
\begin{equation*}
E_{\chi _{n+4}^{2}}[ \frac{1}{( u+c)^{2}}] \leq \frac{1}{( c+n+2)}\frac{1}{( n+c)}\leq \frac{1}{(n+c)^{2}}.
\end{equation*}
\end{proof}

\subsection{A Modified Bayes Estimator}
Now, let $X/\theta \sim N_{p}(\theta, \sigma^{2}I_{p}) $ where $\sigma^{2}$ unknown and estimated by the statistic $S^{2}\sim \sigma^{2}\chi_{n}^{2}$ and $\theta $ has a prior distribution $\theta \sim N_{p}(\nu, \tau^{2}I_{p}) $ with the hyperparameters $\nu, \tau^{2}$ are known.\\
\\
\begin{proposition} \label{p 3.1}
The statistics $\frac{S^{2}}{S^{2}+n \tau^{2}}\ $ is an asymptotically unbiased estimator of the ratio $\frac{\sigma^{2}}{\tau^{2}+\sigma^{2}}$.
\end{proposition}
\begin{proof}
As $S^{2}\sim \sigma^{2}\chi_{n}^{2}$, then
\begin{eqnarray*}
 E(\frac{S^{2}}{S^{2}+n\tau^{2}})&=&E(\frac{\frac{S^{2}}{\sigma^{2}}}{\frac{S^{2}}{\sigma^{2}}+n\frac{\tau^{2}}{\sigma^{2}}})\\
&=&E_{\chi_{n}^{2}}(\frac{u}{u+n\frac{\tau^{2}}{\sigma^{2}}})\\
&=&n E_{\chi_{n+2}^{2}}(\frac{1}{u+n\frac{\tau^{2}}{\sigma^{2}}}),
\end{eqnarray*}
the last equality comes from Lemma \ref{l 5.2} of the Appendix.\\
 From (\ref{eq3.1}) of Lemma \ref{l 3.2}, we have
\begin{equation*}
\frac{n}{n(1+\frac{\tau^{2}}{\sigma^{2}})+2}\leq E(\frac{S^{2}}{S^{2}+n\tau^{2}})=n E_{\chi_{n+2}^{2}}
(\frac{1}{u+n\frac{\tau^{2}}{\sigma^{2}}}) \leq \frac{n}{n(1+\frac{\tau^{2}}{\sigma^{2}})},
\end{equation*}
thus
\begin{equation*}
\lim_{n\rightarrow \infty }E(\frac{S^{2}}{S^{2}+n\tau^{2}})=\frac{1}{1+\frac{\tau^{2}}{\sigma^{2}}}=\frac{\sigma^{2}}{\sigma^{2}+\tau^{2}}.
\end{equation*}
\end{proof}

If we replace in formula (\ref{eq2.1}) the ratio $\frac{\sigma^{2}}{\sigma^{2}+\tau^{2}} $ by its estimator $\frac{S^{2}}{S^{2}+n\tau^{2}}$, we obtain a Modified Bayes estimator expressed as
\begin{eqnarray}\label{eq3.7}
\delta_{B}^{\ast}=(1-\frac{S^{2}}{S^{2}+n\tau^{2}})(X-\nu )+\nu
\end{eqnarray}

The following Theorem gives an explicit formula of the risk of the Modified Bayes estimator $\delta_{B}^{\ast}$, a lower and upper bound of risks ratio of the estimator $\delta_{B}^{\ast}$ to the maximum likelihood estimator $X$.\\
\begin{theorem} \label{t 3.1}
Let the Modified Bayes estimator $\delta_{B}^{\ast}$ given in (\ref{eq3.7}), then\\
i) the quadratic risk of the estimator $\delta_{B}^{\ast}$ is
\begin{equation*}
R(\delta_{B}^{\ast};\nu, \tau^{2}, \sigma^{2})=p\sigma^{2}{1+n(n+2)(1+\frac{\tau^{2}}{\sigma^{2}}) E_{\chi_{n+4}^{2}}(\frac{1}{(u+n\frac{\tau^{2}}{\sigma^{2}})^{2}})-2nE_{\chi_{n+2}^{2}}(\frac{1}{u+n\frac{\tau^{2}}{\sigma^{2}}})},
\end{equation*}\\
ii)
\begin{equation*}
1+\frac{n(n+2) (1+\frac{\tau^{2}}{\sigma^{2}})}{(n(1+\frac{\tau^{2}}{\sigma^{2}})+4)^{2}}-\frac{2}{1+\frac{\tau^{2}}{\sigma^{2}}}\leq \frac{R(\delta_{B}^{\ast }; \nu, \tau^{2}, \sigma^{2})}{R( X)}\leq 1+\frac{(n+2)}{n(1+\frac{\tau^{2}}{\sigma^{2}})}-\frac{2n}{n(1+\frac{\tau^{2}}{\sigma^{2}})+2}.
\end{equation*}
\end{theorem}
\begin{proof}
i) We have
\begin{equation*}
R(\delta_{B}^{\ast}, \theta )=E_{\theta}\Big\|\Big(1-\frac{S^{2}}{S^{2}+n\tau^{2}}\Big)(X-\nu ) +\nu -\theta\Big\|^{2}.
\end{equation*}
From the independence of variables $X$ and $S^{2}$, we have
\begin{equation*}
\begin{split}
R(\delta_{B}^{\ast}, \theta)&=E_{\theta}\| X-\theta\|^{2}+E_{\theta}\Big(\frac{S^{2}}{S^{2}+n\tau^{2}}\Big)^{2}E_{\theta}\| X-\nu\| ^{2}
\\&-2\sigma^{2}E_{\theta}\Big(\frac{S^{2}}{S^{2}+n\tau^{2}}\Big) E_{\theta}\Big[\langle \frac{X-\theta}{\sigma}, \frac{X-\nu}{\sigma}\rangle\Big]
\\&=p\sigma^{2}+E_{\chi_{n}^{2}}\Big(\frac{u}{u+n\frac{\tau^{2}}{\sigma^{2}}}\Big)^{2}(p\sigma^{2}+\|\theta -\nu\|^{2})-2p\sigma^{2}E_{\chi_{n}^{2}}\Big(\frac{u}{u+n\frac{\tau^{2}}{\sigma^{2}}}\Big)
\\&=p\sigma^{2}+n(n+2)E_{\chi_{n+4}^{2}}\Big(\frac{1}{(u+n\frac{\tau^{2}}{\sigma^{2}})^{2}}\Big)(p\sigma^{2}+\|\theta-\nu \|^{2})
\\&-2pn\sigma^{2}E_{\chi_{n+2}^{2}}\Big(\frac{1}{u+n\frac{\tau^{2}}{\sigma^{2}}}\Big),
\end{split}
\end{equation*}
\begin{equation*}
\begin{split}
 R\left(\delta_{B}^{\ast}; \nu, \tau^{2}, \sigma^{2}\right)&=E_{\nu, \tau^{2}, \sigma^{2}}\left[R\left(\delta_{B}^{\ast};\theta \right) \right]
\\&=p\sigma^{2}+n(n+2)E_{\chi_{n+4}^{2}}\left(\frac{1}{u+n\frac{\tau^{2}}{\sigma^{2}}}\right)^{2}\left(p\sigma^{2}+E_{\nu, \tau^{2}}\|\theta-\nu \|^{2}\right)\\&-2pn\sigma^{2}E_{\chi_{n+2}^{2}}\left(\frac{1}{u+n\frac{\tau^{2}}{\sigma^{2}}}\right)
\end{split}
\end{equation*}
\begin{equation*}
\hspace{2.5cm}=p\sigma^{2}\left\{1+n(n+2)\left(1+\frac{\tau^{2}}{\sigma^{2}}\right)E_{\chi_{n+4}^{2}}\left(\frac{1}{\left(u+n\frac{\tau^{2}}
{\sigma^{2}}\right)^{2}}\right)-2nE_{\chi_{n+2}^{2}}\left(\frac{1}{u+n\dfrac{\tau^{2}}{\sigma^{2}}}\right) \right\} .
\end{equation*}
ii) From i), we have
\begin{equation*}
\frac{R\left(\delta_{B}^{\ast}; \nu, \tau^{2}, \sigma^{2}\right)}{R\left( X\right)}=1+n(n+2)\left(1+\frac{\tau^{2}}{\sigma^{2}}\right)E_{\chi_{n+4}^{2}}\left(\frac{1}{\left(u+n\dfrac{\tau^{2}}{\sigma^{2}}\right)^{2}}\right)
-2nE_{\chi_{n+2}^{2}}\left(\frac{1}{u+n\frac{\tau^{2}}{\sigma^{2}}}\right) .
\end{equation*}
By using formulas (\ref{eq3.1}) and (\ref{eq3.2}) of Lemma \ref{l 3.2}, we obtain
\begin{eqnarray*}
\frac{R\left(\delta_{B}^{\ast}; \nu , \tau^{2}, \sigma^{2}\right)}{R\left( X\right)} &\leq &1+n(n+2)\left(1+\frac{\tau^{2}}{\sigma^{2}}\right) \frac{1}{\left(n+n\frac{\tau^{2}}{\sigma^{2}}\right)^{2}}-\frac{2n}{n+2+n\frac{\tau^{2}}{\sigma^{2}}} \\
&\leq &1+\frac{(n+2)}{n\left(1+\frac{\tau^{2}}{\sigma^{2}}\right)}-\frac{2n}{n\left(1+\frac{\tau^{2}}{\sigma^{2}}\right)+2}
\end{eqnarray*}
and
\begin{equation*}
\frac{R\left(\delta_{B}^{\ast}; \nu, \tau^{2}, \sigma^{2}\right)}{R\left(X\right)}\geq 1+n(n+2)\left(1+\frac{\tau^{2}}{\sigma^{2}}\right)\frac{1}{\left( n+4+n\frac{\tau^{2}}{\sigma^{2}}\right)^{2}}-\frac{2n}{n\left(1+\frac{\tau^{2}}{\sigma^{2}}\right)}
\end{equation*}
\begin{equation*}
\hspace{-0.5cm} \geq 1+\frac{n(n+2)\left(1+\frac{\tau^{2}}{\sigma^{2}}\right)}{\left(n\left(1+\frac{\tau^{2}}{\sigma^{2}}\right)+4\right)^{2}}
-\frac{2}{1+\frac{\tau^{2}}{\sigma^{2}}}.
\end{equation*}
\end{proof}

\begin{theorem} \label{t 3.2}
a) If $n\geq 5$, the estimator $\delta_{B}^{\ast}\ $ given in (\ref{eq3.7}) is minimax,\\
b) $\underset{n,p\rightarrow \infty }{\lim }\dfrac{R\left( \delta_{B}^{\ast}; \nu, \tau^{2}, \sigma^{2}\right)}{R\left( X\right)}=\dfrac{\tau^{2}}{\tau^{2}+\sigma^{2}}$.
\end{theorem}
\begin{proof}
a) From the previous Theorem, we have
\begin{eqnarray}\label{eq3.8}
R\left(\delta_{B}^{\ast}; \nu, \tau^{2}, \sigma^{2}\right) \leq p\sigma^{2}\left\{1+n(n+2)\left(1+\frac{\tau^{2}}{\sigma^{2}}\right) \left(\frac{1}{\left(n+n\dfrac{\tau^{2}}{\sigma^{2}}\right)^{2}}\right)-2n\left(\frac{1}{\left(n+2+n\dfrac{\tau^{2}}{\sigma^{2}}\right)}\right) \right\}.
\end{eqnarray}
The study of the variation of the real function $h\left(x\right)=\dfrac{(x+2)}{x\left(1+\dfrac{\tau^{2}}{\sigma^{2}}\right)}-\dfrac{2x}{x\left(1+\dfrac{\tau^{2}}{\sigma^{2}}\right)+2}$, shows that
\begin{equation*}
\frac{(n+2)}{n\left(1+\dfrac{\tau^{2}}{\sigma^{2}}\right)}-\frac{2n}{n\left(1+\dfrac{\tau^{2}}{\sigma^{2}}\right)+2}\leq 0\, \, \text{for any }n\geq 5.
\end{equation*}
Then
\begin{equation*}
R\left(\delta_{B}^{\ast}; \nu, \tau^{2}, \sigma^{2}\right) \leq p\sigma^{2}\, \, \text{for\ any}\ n\geq 5.
\end{equation*}
b) Immediately from ii) of the Theorem \ref{t 3.1}.
\end{proof}

\subsection{An Empirical Modified Bayes Estimators}
Now, let $X/\theta \sim N_{p}\left(\theta, \sigma^{2}I_{p}\right)$ where $\sigma^{2}$ unknown and estimated by the statistic $S^{2}\sim \sigma^{2}\chi _{n}^{2}$ and $\theta $ has a prior distribution $\theta \sim N_{p}\left(\nu, \tau^{2}I_{p}\right)$ with the hyperparameter $\nu$ is known and the hyperparameter $\tau^{2}$ is unknown.

\begin{proposition} \label{p 3.2}
The statistics $\dfrac{p-2}{n+2}$ $\dfrac{S^{2}}{\| X-\nu \|^{2}}$ is an asymptotically unbiased estimator of the ratio $\dfrac{\sigma^{2}}{\tau^{2}+\sigma^{2}}$.
\end{proposition}
\begin{proof}
From the independence of variables $S^{2}$ and $\| X\|^{2}$, we have
\begin{eqnarray*}
E\left( \frac{p-2}{n+2}\frac{S^{2}}{\| X-\nu \|^{2}}\right)&=&\frac{p-2}{n+2}E\left( S^{2}\right) E\left( \frac{1}{\| X-\nu \| ^{2}}\right)  \\
&=&\frac{p-2}{n+2}\frac{\sigma^{2}}{\tau^{2}+\sigma^{2}}E\left(\frac{S^{2}}{\sigma^{2}}\right) E\left(\frac{1}{\frac{\| X-\nu \|^{2}}{\tau^{2}+\sigma^{2}}}\right).
\end{eqnarray*}
As the marginal distribution of $X$ is: $X\sim N_{p}\left(\nu, \left(\tau^{2}+\sigma^{2}\right) I_{p}\right)$ and $S^{2}\sim \sigma^{2}\chi_{n}^{2}$, we obtain
\begin{equation*}
E\left(\frac{S^{2}}{\sigma^{2}}\right)=n\,\text{ and }\frac{\| X-\nu \|^{2}}{\tau^{2}+\sigma^{2}}\sim \chi_{p}^{2}.
\end{equation*}
Using the Definition \ref{de 2.1}, we have
\begin{equation*}
E\left(\frac{1}{\frac{\|X-\nu \|^{2}}{\tau^{2}+\sigma^{2}}}\right)=E\left(\dfrac{1}{\chi_{p}^{2}}\right)=\dfrac{1}{p-2},
\end{equation*}
thus
\begin{equation*}
E\left(\frac{p-2}{n+2}\frac{S^{2}}{\| X-\nu \|^{2}} \right)=\frac{\sigma^{2}}{\tau^{2}+\sigma^{2}}\frac{n}{n+2}\underset{n\rightarrow \infty }{\rightarrow }\frac{\sigma^{2}}{\tau^{2}+\sigma^{2}}.
\end{equation*}
\end{proof}

If we replace in formula (\ref{eq2.1}) the ratio $\dfrac{\sigma^{2}}{\sigma^{2}+\tau^{2}}$ by its estimator $\dfrac{p-2}{n+2}\dfrac{S^{2}}{\| X-\nu \|^{2}}$, we obtain an Empirical Modified Bayes estimator expressed as
\begin{eqnarray}\label{eq3.9}
\delta_{EB}^{\ast}=\left(1-\frac{p-2}{n+2}\frac{S^{2}}{\| X-\nu \|^{2}}\right) \left(X-\nu \right)+\nu.
\end{eqnarray}
The following Theorem gives an explicit formula of the risk of the Empirical Modified Bayes estimator $\delta_{EB}^{\ast}$.\\
\begin{theorem} \label{t 3.3}
The quadratic risk of the Empirical Bayes estimator $ \delta_{EB}^{\ast}$ given in (\ref{eq3.9}) is
\begin{equation*}
R\left(\delta_{EB}^{\ast }; \nu, \tau^{2}, \sigma^{2} \right)=p\sigma^{2}\left\{1-\frac{p-2}{p}\frac{n}{n+2}\frac{\sigma^{2}}{\tau^{2}+\sigma^{2}}\right\}.
\end{equation*}
\end{theorem}
\begin{proof}
\begin{equation*}
\begin{split}
R\left(\delta_{EB}^{\ast }; \nu , \tau^{2}, \sigma^{2}\right)&=E\left(\|\left(1-\frac{p-2}{n+2}\frac{S^{2}}{\| X-\nu \|^{2}}\right) \left(X-\nu \right)+\nu -\theta \|^{2}\right)
\\&=E\left(\| X-\theta \|^{2}\right)+\left(\frac{p-2}{n+2}\right)^{2}\frac{\sigma^{4}}{\tau^{2}+\sigma^{2}}E\left(\frac{S^{2}}{\sigma^{2}}\right)^{2}E\left(\frac{1}{\frac{\| X-\nu \|^{2}}{\tau^{2}+\sigma^{2}}}\right)  \\
&-\frac{2\left( p-2\right)}{n+2}\sigma^{2}E\left(\frac{S^{2}}{\sigma^{2}}\right) E\left(\left\langle \frac{X-\theta}{\sigma}, \frac{1}{\frac{\| X-\nu \|^{2}}{\sigma^{2}}}\left(\frac{X-\nu }{\sigma}\right) \right\rangle \right) .
\end{split}
\end{equation*}
As $\dfrac{\| X-\theta \|^{2}}{\sigma^{2}}\sim \chi_{p}^{2}$, the marginal distribution of $X$ is: $X\sim N_{p}\left(\nu,\left(\tau^{2}+\sigma^{2}\right) I_{p}\right)$ and $S^{2}\sim \sigma^{2}\chi_{n}^{2}$, we obtain
\begin{equation*}
E\left(\|X-\theta \|^{2}\right)=p\sigma^{2},\, E\left(\frac{S^{2}}{\sigma^{2}}\right)=n\,\text{ and }\frac{\| X-\nu \|^{2}}{\tau^{2}+\sigma^{2}}\sim \chi_{p}^{2}.
\end{equation*}
Using the Definition \ref{de 2.1}, we have
\begin{equation*}
E\left(\frac{S^{2}}{\sigma^{2}}\right)^{2}=n\left( n+2\right)
\end{equation*}
 and
\begin{equation*}
E\left(\frac{1}{\frac{\| X-\nu \|^{2}}{\tau^{2}+\sigma^{2}}}\right)=E\left(\dfrac{1}{\chi_{p}^{2}}\right)=\dfrac{1}{p-2}.
\end{equation*}
Let $Y=\left(y_{1},y_{2},...,y_{p}\right)^{t}=\dfrac{X-\theta }{\sigma}$. It is clear that $Y/\theta \sim N_{p}\left(0,I_{p}\right)$, thus
\begin{eqnarray*}
E\left(\left\langle \frac{X-\theta}{\sigma }, \frac{1}{\frac{\|X-\nu \|^{2}}{\sigma^{2}}}\left(\frac{X-\nu }{\sigma}\right)
\right\rangle \right)&=&E\left(\left\langle Y, \frac{1}{\| Y+\frac{\theta -\nu }{\sigma}\|^{2}}\left( Y+\frac{\theta -\nu }{\sigma}\right) \right\rangle \right)  \\
&=&\sum_{i=1}^{p}E\left[ y_{i}\left(\frac{1}{\sum_{j=1}^{p}\left( y_{j}+\frac{\theta_{j}-\nu_{j}}{\sigma }\right)^{2}}\left( y_{i}+\frac{\theta_{i}-\nu _{i}}{\sigma }\right) \right) \right].
\end{eqnarray*}
Using the Lemma \ref{l 5.1} (Appendix), we have
\begin{eqnarray*}
E\left(\left\langle Y,\frac{1}{\| Y+\frac{\theta -\nu }{\sigma} \|^{2}}\left( Y+\frac{\theta -\nu }{\sigma}\right) \right\rangle \right)  &=&\sum_{i=1}^{p}E\left[\frac{\partial}{\partial y_{i}}\left(\frac{1}{\overset{p}{\underset{j=1}{\sum}}\left( y_{j}+\frac{\theta_{j}-\nu_{j}}{\sigma} \right)^{2}}\left(y_{i}+\frac{\theta_{i}-\nu_{i}}{\sigma }\right) \right) \right]  \\
&=&\sum_{i=1}^{p}E\left[ \frac{1}{\overset{p}{\underset{j=1}{\sum}}\left( y_{j}+\frac{\theta_{j}-\nu_{j}}{\sigma }\right)^{2}}-\frac{2\left( y_{i}+\frac{\theta_{i}-\nu_{i}}{\sigma}\right)^{2}}{\left[\overset{p}{\underset{j=1}{\sum}}\left( y_{j}+\frac{ \theta_{j}-\nu_{j}}{\sigma }\right)^{2}\right]^{2}}\right]  \\
&=&pE\left[\frac{1}{\overset{p}{\underset{j=1}{\sum}}\left( y_{j}+\frac{\theta_{j}-\nu_{j}}{\sigma}\right)^{2}}\right]-2E\left[\frac{\overset{p}{\underset{i=1}{\sum}}\left( y_{i}+\frac{\theta _{i}-\nu_{i}}{\sigma}\right)^{2}}{\left[\overset{p}{\underset{j=1}{\sum}}\left(y_{j}+\frac{\theta_{j}-\nu_{j}}{\sigma}\right)^{2}\right]^{2}}\right]  \\
&=&\left( p-2\right) E\left[\frac{1}{\overset{p}{\underset{j=1}{\sum}}\left( y_{j}+\frac{\theta_{j}-\nu_{j}}{\sigma}\right)^{2}}\right]  \\
&=&\left( p-2\right) E\left[\frac{1}{\left\Vert Y+\frac{\theta -\nu}{\sigma }\right\Vert^{2}}\right]  \\
&=&\left( p-2\right) \frac{\sigma^{2}}{\tau^{2}+\sigma^{2}}E\left[\frac{1}{\frac{\left\Vert X-\nu \right\Vert^{2}}{\tau^{2}+\sigma^{2}}}\right]  \\
&=&\frac{\sigma^{2}}{\tau^{2}+\sigma^{2}}.
\end{eqnarray*}
The last equality follows from the fact that $E\left(\frac{1}{\frac{\left\Vert X-\nu \right\Vert^{2}}{\tau^{2}+\sigma^{2}}}\right) =\dfrac{1}{p-2}$. \\
Hence
\begin{equation*}
R\left(\delta_{EB}^{\ast }; \nu, \tau^{2},\sigma^{2}\right)=p\sigma^{2}\left\{1-\frac{p-2}{p}\frac{n}{n+2}\frac{\sigma^{2}}{\tau^{2}+\sigma^{2}}\right\}.
\end{equation*}
\end{proof}
\begin{theorem} \label{t 3.4}
Let the Empirical Bayes estimator $\delta_{EB}^{\ast }$ given in (\ref{eq3.9}), then\\
a) If $p\geq 3$, the estimator $\delta_{EB}^{\ast }$ is minimax,\\
b) $\underset{n,p\rightarrow \infty }{\lim }\dfrac{R\left( \delta_{EB}^{\ast };\nu, \tau^{2}, \sigma^{2}\right)}{R\left( X\right)}=\dfrac{\tau ^{2}}{\tau^{2}+\sigma^{2}}$.
\end{theorem}
\begin{proof}
Immediately from the Theorem \ref{t 3.3}.
\end{proof}

Now, we consider the class of shrinkage estimators defined by:
\begin{eqnarray}\label{eq3.10}
\delta _{EB}^{\ast c}=\left( 1-c\frac{S^{2}}{\| X-\nu \|^{2}}\right) \left( X-\nu \right) +\nu
\end{eqnarray}
where $c$ is a real may be depend on $n$ and $p$.
\begin{proposition} \label{p 3.3}
The quadratic risk of the  estimator $\delta _{EB}^{\ast c}$ given in (\ref{eq3.10}) is
\begin{equation*}
R\left( \delta _{EB}^{\ast c};\nu ,\tau ^{2},\sigma ^{2}\right) =p\sigma
^{2}\left\{ 1-\frac{2n\left( p-2\right) c-n\left( n+2\right) c^{2}}{p-2}
\frac{\sigma ^{2}}{\tau ^{2}+\sigma ^{2}}\right\} .
\end{equation*}
\end{proposition}
\begin{proof}
Analogous to the proof of Theorem \ref{t 3.3}, so we give a brief idea.
\begin{eqnarray*}
R\left( \delta _{EB}^{\ast };\nu ,\tau ^{2},\sigma ^{2}\right)  &=&E\left(
\|\left( 1-c\frac{S^{2}}{\|X-\nu \| ^{2}}\right)\left( X-\nu \right) +\nu -\theta \|^{2}\right)  \\
&=&E\left(\| X-\theta \| ^{2}\right) +c^{2}\frac{\sigma^{4}}{\tau ^{2}+\sigma ^{2}}E\left( \frac{S^{2}}{\sigma ^{2}}\right)
^{2}E\left( \frac{1}{\frac{\| X-\nu \| ^{2}}{\tau^{2}+\sigma ^{2}}}\right)  \\
&&-2c\sigma ^{2}E\left( \frac{S^{2}}{\sigma ^{2}}\right) E\left(\left\langle \frac{X-\theta }{\sigma },\frac{1}{\frac{\| X-\nu
\|^{2}}{\sigma ^{2}}}\left( \frac{X-\nu }{\sigma }\right)\right\rangle \right) .
\end{eqnarray*}
Using the same technicals of Theorem \ref{t 3.3}, we obtain
\begin{equation*}
\left( \delta _{EB}^{\ast };\nu ,\tau ^{2},\sigma ^{2}\right)=p\sigma ^{2}\left\{ 1-\frac{2n\left( p-2\right) c-n\left( n+2\right) c^{2}
}{p-2}\frac{\sigma ^{2}}{\tau ^{2}+\sigma ^{2}}\right\} .
\end{equation*}
\end{proof}
\begin{theorem} \label{t 3.5}
Let the estimator $\delta _{EB}^{\ast c}$ given in (\ref{eq3.10}), then\\
a) If $p\geq 3,$ a sufficient condition that the estimator
$\delta _{EB}^{\ast c}$ is minimax, is
\begin{equation*}
0\leq c\leq \frac{2\left( p-2\right) }{n+2},
\end{equation*}\\
b) the optimal value of $c$ so that the risk of the estimator $\delta _{EB}^{\ast c}$ is minimal is $\widehat{c}=\dfrac{\left( p-2\right) }{n+2}$, thus the estimator $\delta _{EB}^{\ast }$ is the best in the class of shrinkage estimators $\delta _{EB}^{\ast c}.$
\end{theorem}
\begin{proof}
a) Using the Proposition \ref{p 3.3}, a sufficient condition that
the estimator $\delta _{EB}^{\ast c}$ is minimax, is
\begin{equation*}
n\left( n+2\right) c^{2}-2n\left( p-2\right) c\leq 0,
\end{equation*}
hence
\begin{equation*}
0\leq c\leq \frac{2\left( p-2\right) }{n+2}.
\end{equation*}
b) From the convexity of the risk function $R\left( \delta_{EB}^{\ast c};\nu ,\tau ^{2},\sigma ^{2}\right) $ as a function of $c$, the
optimal value of $c$ so that the risk of the estimator $\delta _{EB}^{\ast c}$ is minimal is $\widehat{c}=\dfrac{\left( p-2\right)}{n+2}$. thus the
estimator $\delta _{EB}^{\ast }$ is the best in the class of estimators $\delta _{EB}^{\ast c}.$
\end{proof}
\subsection{\textit{Simulation}}
First, we illustrate graphically the risk ratios of the Modified Bayes estimator $\delta_{B}^{\ast}$ to the maximum likelihood estimator $X$, $\dfrac{R\left(\delta_{B}^{\ast }; \nu, \tau^{2}, \sigma^{2}\right)}{R\left(X\right)}$ as a function of $\lambda=\dfrac{\tau^{2}}{\sigma^{2}}$ for various values of $n$.
\begin{figure}[!htbp]
\centering
\includegraphics[width=9cm]{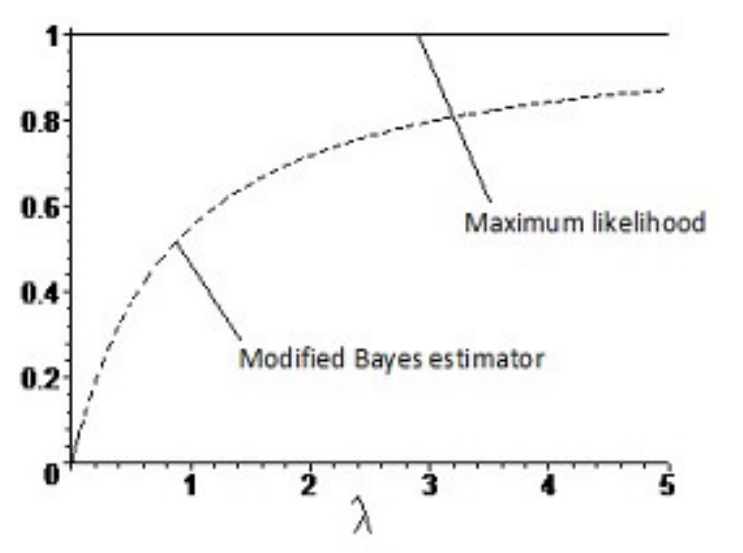}
\caption{$n=5$}
\label{ContinuationFig.1}
\end{figure}

\begin{figure}[!htbp]
\centering
\includegraphics[width=9cm]{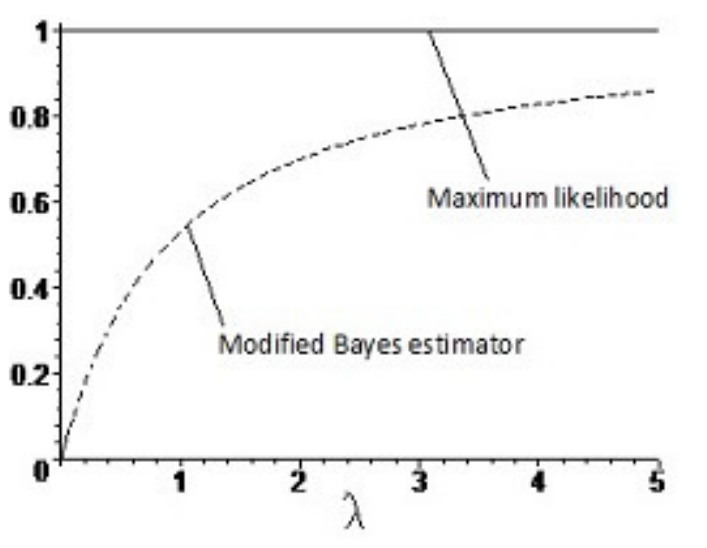}
\caption{$n=8$}
\label{ContinuationFig.2}
\end{figure}

In Fig.1 and Fig.2, we note that the risks ratio of the Modified Bayes estimator $\delta_{B}^{\ast}$ to the maximum likelihood estimator $X$, is less than 1, thus the Modified Bayes estimator is minimax for $n=5$ and $n=8$.\\

Secondly, we illustrate graphically the risks difference $ \vartriangle_{R}=R\left(\delta_{B}^{\ast }; \nu, \tau^{2},\sigma^{2}\right)-R\left(X\right)$ of the Modified Bayes estimator $\delta_{B}^{\ast}$ and the maximum likelihood estimator $X$, as a function of $ x=\tau^{2}$ and $ y=\sigma^{2}$ for various values of $n$.
\begin{figure}[!htbp]
\centering
\includegraphics[width=9cm]{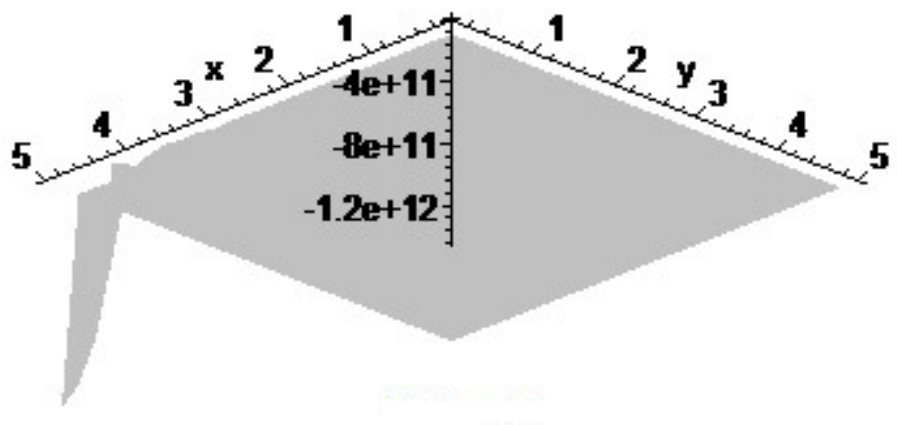}
\caption{$n=22$}
\label{ContinuationFig.3}
\end{figure}

\begin{figure}[!htbp]
\centering
\includegraphics[width=9cm]{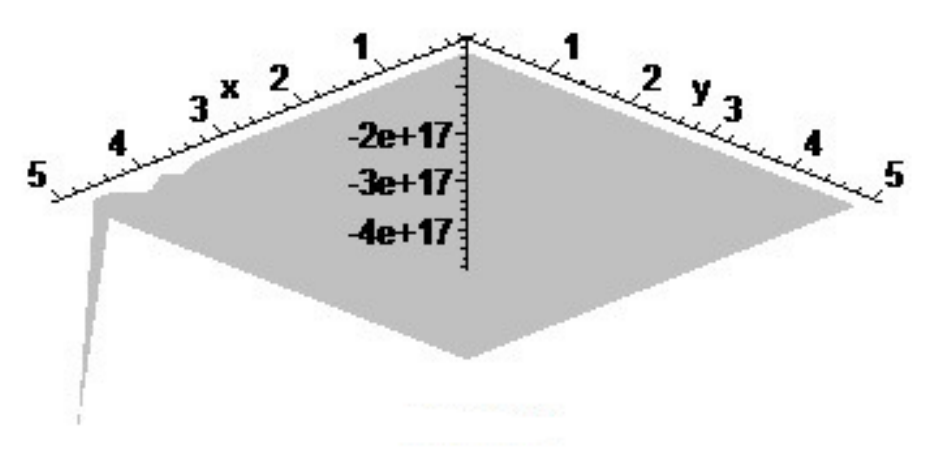}
\caption{$n=28$}
\label{ContinuationFig.4}
\end{figure}
In Fig.3 and Fig.4, we note that the risks difference between the Modified Bayes estimator $\delta_{B}^{\ast}$ and the Maximum likelihood estimator $X$, is negative, thus the Modified Bayes estimator is minimax for $n=22$ and $n=28$.\\

Finally, we illustrate the graphs of the upper bound given by the formula (3.8) for the risks difference  $ \vartriangle_{R}=R\left(\delta_{B}^{\ast }; \nu, \tau^{2},\sigma^{2}\right)-R\left(X\right)$ divised by the risk of the maximum likelihood estimator $R\left(X\right)$, as a function of $y=\dfrac{\tau^{2}}{\sigma^{2}}$ for various values of $n$.
\begin{figure}[H]
\centering
\includegraphics[width=7cm]{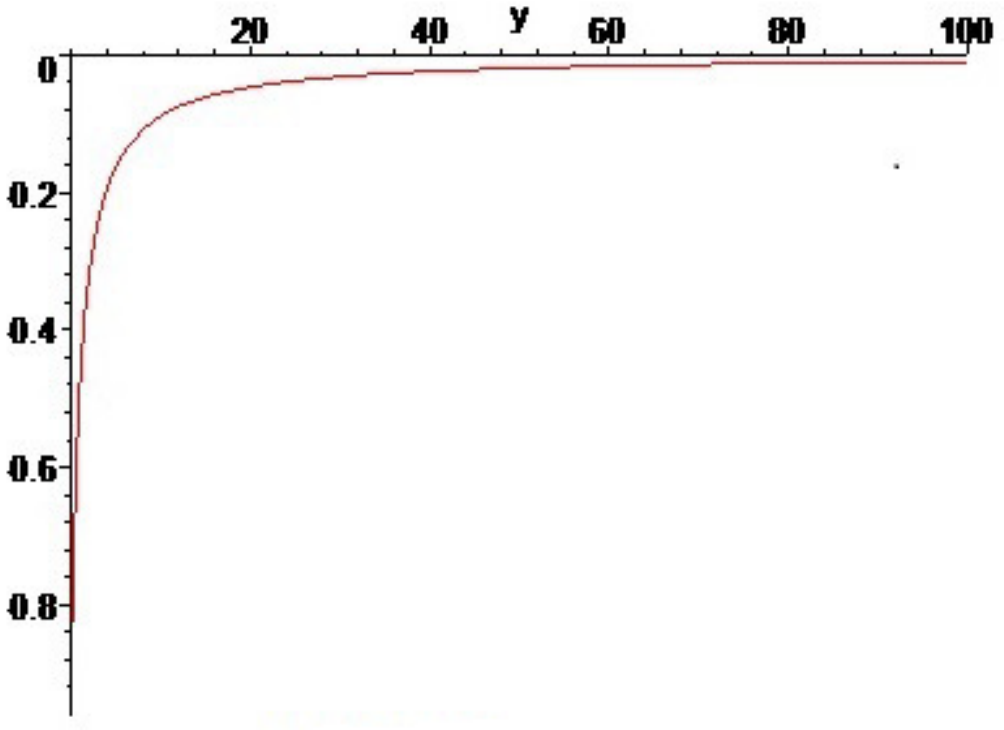}
\caption{$n=100$}
\label{ContinuationFig.5}
\end{figure}
 
\begin{figure}[H]
\centering
\includegraphics[width=7cm]{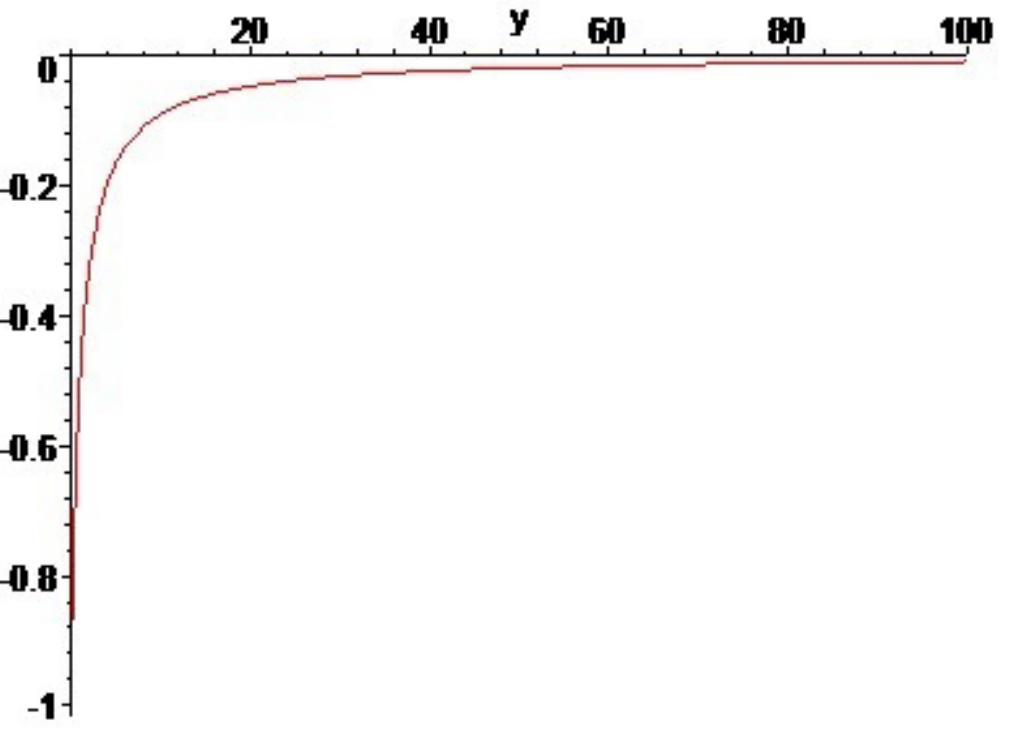}
\caption{$n=1000$}
\label{ContinuationFig.6}
\end{figure}
In Fig.5 and Fig.6, we note that the risks difference between the Modified Bayes estimator $\delta_{B}^{\ast}$ and the Maximum likelihood estimator $X$, is negative, thus the Modified Bayes estimator is minimax for the large values of $n$, for example $n=100$ and $n=1000$.

\section{\textbf{Conclusion}}
In this work, we studied the shrinkage estimators of a multivariate normal mean distributon in the Bayesian case. We considered the model $X\sim N_{p}\left(\theta, \sigma^{2}I_{p}\right) $ with $\sigma^{2}$ is unknown and we take the prior distribution $\theta \sim N_{p}\left(\upsilon, \tau^{2}I_{p}\right)$ where the hyperparameter $\upsilon$ is known and the hyperparameter $\tau^{2}$ is known or unknown. We constructed a Modified Bayes estimator $\delta_{B}^{\ast}$ when the hyperparameter $\tau^{2}$ is known and an Empirical Modified Bayes estimator $\delta_{EB}^{\ast}$ when the hyperparameter $\tau^{2}$ is unknown. We showed that the estimators $\delta_{B}^{\ast}$ and $\delta_{EB}^{\ast}$ are minimax when $n$ and $p$ are finite. When $n$ and $p$ tend simultaneously to infinity, the results agree with the one obtained in our previous published papers. An extension of this work is to study the minimaxity and the limits of risks ratios, when the model and the prior law have both a symmetrical spherical distribution.

 \section{Appendix}
In this section we give the following lemmas cited in \cite{s81} and \cite{ch82} respectively.
\begin{lemma} \label{l 5.1}
Let $Y$ be a $N(0, 1)$ real random variable and let $g: \R \longrightarrow \R $ be an indefinite integral of the Lebesgue measurable function $g^{\prime}$, essentially the derivative of $g.$ Suppose also that $E|g^{\prime}(Y)| <+\infty$. Then
\begin{equation*}
E[Yg(Y)]=E(g^{\prime}(Y)).
\end{equation*}
\end{lemma}
\begin{lemma} \label{l 5.2}
For any real function $h$ such that $E(h(\chi_{q}^{2}(\lambda )) \chi_{q}^{2}(\lambda)) $ exists, we have
\begin{equation*}
E{h(\chi_{q}^{2}(\lambda)) \chi_{q}^{2}(\lambda)}=qE{h(\chi_{q+2}^{2}(\lambda))}+2\lambda E{h(\chi_{q+4}^{2}(\lambda))} .
\end{equation*}
\end{lemma}

\end{document}